\documentclass[11pt]{amsart}

\usepackage{amsmath, amssymb, amsfonts, amsthm, verbatim, graphicx}
\usepackage{epstopdf, mathtools}
\mathtoolsset{showonlyrefs}
\usepackage{hyperref}

\newtheorem{theorem}{Theorem}[section]
\newtheorem{proposition}[theorem]{Proposition}
\newtheorem{corollary}[theorem]{Corollary}
\newtheorem{lemma}[theorem]{Lemma}
\newtheorem{definition}[theorem]{Definition}
\theoremstyle{definition} \newtheorem{remark}[theorem]{Remark}

\numberwithin{figure}{section}

\newcommand{\bR}{{\mathbb R}}
\newcommand{\bN}{{\mathbb N}}
\newcommand{\bC}{{\mathbb C}}

\newcommand{\bZ}{{\mathbb Z}}

\newcommand{\bP}{{\mathbb P}}

\newcommand{\cF}{{\mathcal{F}}}

\newcommand{\cI}{{\mathcal{I}}}

\newcommand{\cH}{{\mathcal{H}}}

\begin{document}

\title[]{On the almost sure global well-posedness of energy sub-critical nonlinear wave equations on $\bR^3$}
\begin{abstract}
We consider energy sub-critical defocusing nonlinear wave equations on $\bR^3$ and establish the existence of unique global solutions almost surely with respect to a unit-scale randomization of the initial data on Euclidean space.  In particular, we provide examples of initial data at super-critical regularities which lead to unique global solutions. The proof is based on probabilistic growth estimates for a new modified energy functional. This work improves upon the authors' previous results in \cite{LM} by significantly lowering the regularity threshold and strengthening the notion of uniqueness.
\end{abstract}

\author[]{Jonas L\"uhrmann}
\address{Departement Mathematik \\ ETH Z\"urich \\ 8092 Z\"urich \\ Switzerland}
\email{jonas.luehrmann@math.ethz.ch}

\author[]{Dana Mendelson}
\address{Department of Mathematics \\ Massachusetts Institute of Technology \\ 77 Massachusetts Ave \\ Cambridge, MA 02139 \\ USA}
\email{dana@math.mit.edu}

\thanks{\textit{2010 Mathematics Subject Classification.} 35L05, 35R60, 35Q55}
\thanks{\textit{Key words and phrases.} nonlinear wave equation; almost sure global well-posedness; random initial data}
\thanks{The first author was supported in part by the Swiss National Science Foundation under grant SNF 200020-159925. The second author was supported in part by the U.S. National Science Foundation grant DMS-1362509.}

\maketitle

\section{Introduction} \label{sec:introduction}
\setcounter{equation}{0}

In this note we consider the Cauchy problem for the defocusing nonlinear wave equation
\begin{equation} \label{equ:ivp}
 \left\{ \begin{aligned}
  -\partial_t^2 u + \Delta u &= |u|^{p-1} u \text{ on } \bR \times \bR^3, \\
  (u, \partial_t u)|_{t=0} &= (f_1, f_2) \in H^s_x(\bR^3) \times H^{s-1}_x(\bR^3),
 \end{aligned} \right.
\end{equation}
where $3 < p < 5$ and $H^s_x(\bR^3)$ is the usual inhomogeneous Sobolev space. The main result of this paper establishes the almost sure existence of global unique solutions to~\eqref{equ:ivp} with respect to a unit-scale randomization of initial data in $H^s_x(\bR^3) \times H^{s-1}_x(\bR^3)$ for $3 < p < 5$ and $\frac{p-1}{p+1} < s < 1$. In particular, for the entire range of exponents $p$ we obtain large sets of super-critical initial data that lead to unique global solutions. This improves over the authors' previous result in \cite{LM} both regarding the threshold for allowable regularities and the notion of uniqueness of the solutions. Our proof relies on probabilistic growth estimates for a new modified energy functional which we introduce in \eqref{equ:energy}. These estimates are inspired by the recent work of Oh and Pocovnicu \cite{OP} on the almost sure global existence of solutions to the quintic nonlinear wave equation on $\bR^3$.

\medskip

The equation \eqref{equ:ivp} is invariant under the scaling transformation
\[
 u(t,x) \mapsto u_\lambda(t,x) = \lambda^{\frac{2}{p-1}} u(\lambda t, \lambda x) \quad \text{for } \lambda > 0,
\]
which gives rise to the scaling invariant critical regularity $s_c = \frac{3}{2} - \frac{2}{p-1}$. In \cite{Lindblad_Sogge}, Lindblad and Sogge construct local strong solutions to \eqref{equ:ivp} for sub-critical and critical regularities $s \geq s_c$ using Strichartz estimates for the wave equation. When $s < s_c$, this is the super-critical regime and the well-posedness arguments based on Strichartz estimates break down. Global solutions to \eqref{equ:ivp} were constructed by Kenig, Ponce and Vega \cite[Theorem 1.2]{KPV} using a high-low argument for $2 \leq p < 5$ and initial data in a range of sub-critical spaces below the energy space (see also \cite{BC}, \cite{GP} and \cite{Roy} for $p = 3$).
\medskip

Although it is known that the nonlinear wave equation \eqref{equ:ivp} is ill-posed below the critical scaling regularity (see \cite{L}, \cite{CCT} and \cite{IMM}), using probabilistic tools it is sometimes possible to construct large sets of initial data of super-critical regularity that lead to unique local and even global solutions. This approach was initiated by Bourgain \cite{B94, B96} for the periodic nonlinear Schr\"odinger equation in one and two space dimensions, building upon work by Lebowitz, Rose and Speer \cite{LRS}. Subsequently, Burq and Tzvetkov \cite{BT1, BT2} studied the cubic nonlinear wave equation on a three-dimensional compact manifold by randomizing with respect to an orthonormal eigenbasis of the Laplacian and using invariant measure considerations. Extensive work has been done on such problems in recent years, both in compact and non-compact settings but we restrict the following overview to results in Euclidean space. Many previous results on Euclidean space involve considering a related equation in a setting where an orthonormal basis of eigenfunctions of the Laplacian exists, see for instance \cite{BTT},  \cite{D1}, \cite{Poiret1}, \cite{Poiret2} for results of this type for the defocusing nonlinear Schr\"odinger equation on $\bR^d$ for $d \geq 2$, and \cite{Suzzoni1}, \cite{Suzzoni2} for results of this type for the nonlinear wave equation.

\medskip

It is also possible to randomize initial data directly on Euclidean space using a unit-scale decomposition of frequency space. In several works this has yielded almost sure well-posedness results for super-critical initial data, see for instance \cite{ZF}, \cite{LM}, \cite{BOP1}, \cite{BOP2}, \cite{Pocovnicu} and \cite{OP}. In \cite{LM}, the authors studied the random data problem for \eqref{equ:ivp} and proved almost sure global existence for energy sub-critical nonlinearities. In particular, for $\frac{1}{4}(7+ \sqrt{73}) \simeq 3.89 < p < 5$, the authors obtain almost sure global well-posedness for super-critical initial data. The proof combines a probabilistic local existence argument with Bourgain's high-low frequency decomposition~\cite{B98}, an approach introduced by Colliander and Oh \cite[Theorem 2]{CO} in the context of the one-dimensional periodic defocusing cubic nonlinear Schr\"odinger equation. We note that probabilistic high-low arguments only yield uniqueness in a mild sense, see \cite[Remark 4.3]{LM}, and thus do not provide a definitive answer to the question of uniqueness for super-critical initial data. The high-low method does not extend to energy critical situations, thus there is a natural obstruction to extending the authors' previous results to an energy critical setting. Pocovnicu \cite{Pocovnicu} later proved almost sure global well-posedness for the energy critical defocusing nonlinear wave equation on $\bR^d$ for $d=4,5$. More recently, Oh and Pocovnicu \cite{OP} have treated the energy critical nonlinear wave equation on $\bR^3$. These proofs use probabilistic perturbation theory together with a probabilistic a priori energy bound. See also \cite{BOP2} for a conditional result for the nonlinear Schr\"odinger equation using similar methods.

\subsection{Randomization procedure}

Before stating our main result, we introduce the randomization procedure for the initial data. Let $\psi \in C_c^{\infty}(\bR^3)$ be an even, non-negative function with $\text{supp} (\psi) \subset B(0,1)$ and such that 
\[
 \sum_{k \in \bZ^3} \psi(\xi - k) = 1 \text{ for all } \xi \in \bR^3.
\]
Let $s \in \bR$ and $f \in H^s_x(\bR^3)$. For every $k \in \bZ^3$, we define the function $P_k f: \bR^3 \rightarrow \bC$ by
\begin{equation*}
 (P_k f)(x) = \cF^{-1} \left( \psi(\xi - k) \hat{f}(\xi) \right)(x) \text{ for } x \in \bR^3.
\end{equation*} 
By requiring $\psi$ to be even, we ensure that
\begin{equation} \label{equ:symmetry}
 \overline{P_k f} = P_{-k} f
\end{equation}
for real-valued $f$. As in \cite{LM}, we crucially exploit that these projection operators satisfy a unit-scale Bernstein inequality, namely for all $2 \leq r_1 \leq r_2 \leq \infty$ ,
\begin{align} \label{equ:unit_scale_bernstein}
  \|P_k f\|_{L^{r_2}_x(\bR^3)} \leq C(r_1, r_2) \|P_k f\|_{L^{r_1}_x(\bR^3)}
\end{align}
uniformly for all $f \in L^2_x(\bR^3)$ and $k \in \bZ^3$.  

\medskip

Let now $\{ (h_k, l_k) \}_{k \in \bZ^3}$  be a sequence of zero-mean, complex-valued random variables on a probability space $(\Omega, {\mathcal A}, \bP)$ such that $h_{-k} = \overline {h_k}$ for all $k \in \bZ^3$, and similarly for the $l_k$. We assume that $\{h_0, \textup{Re}(h_k), \textup{Im}(h_k)\}_{k \in \cI}$ are independent, zero-mean, real-valued random variables, where $\cI$ is such that we have a \textit{disjoint} union $\bZ^3 = \cI \cup (-\cI) \cup \{0\}$, and similarly for the $l_k$. Let us denote by $\mu_{k}$ and $\nu_{k}$ the joint distributions of the real and imaginary parts of the $h_k$ and $l_k$, respectively. We assume that there exists $c > 0$ such that
\begin{equation} \label{eq:rvassumption}
 \left| \int_{-\infty}^{+\infty} e^{\gamma x} \, d\mu_{k}(x) \right| \leq e^{c \gamma^2} \text{ for  all } \gamma \in \bR \text{ and for all } k \in \bZ^3,
\end{equation}
and similarly for $\nu_{k}$. The assumption \eqref{eq:rvassumption} is satisfied, for example, by standard Gaussian random variables, standard Bernoulli random variables, or any random variables with compactly supported distributions. 

\medskip

For a given $f = (f_1, f_2) \in H^s_x(\bR^3) \times H^{s-1}_x(\bR^3)$ for some $s \in \bR$, we define its randomization by
\begin{equation} \label{equ:bighsrandomization} 
 f^{\omega} = (f_1^{\omega}, f_2^{\omega}) := \biggl( \sum_{k \in \bZ^3} h_k(\omega) P_k f_1, \sum_{k \in \bZ^3} l_k(\omega) P_k f_2 \biggr),
\end{equation}
where this quantity is understood as a Cauchy limit in $L^2(\Omega; H^s_x(\bR^3) \times H^{s-1}_x(\bR^3))$. The symmetry assumption on the random variables, as well as \eqref{equ:symmetry} ensure that the randomization of real-valued initial data is real-valued. Crucially, such a randomization does not regularize at the level of Sobolev spaces. Similar randomizations have previously been used in \cite{ZF}, \cite{LM}, \cite{BOP1}, \cite{BOP2}, \cite{Pocovnicu}, \cite{OP}. We point out that under the symmetry condition imposed in the randomization,
\begin{align} \label{equ:real_part}
 \qquad f_1^\omega &= \sum_{k \in \bZ^3} h_k(\omega) P_k f_1 = h_0(\omega) P_0 f_1 + 2 \sum_{k \in \cI } \bigl(\textup{Re} \, h_k(\omega) \textup{Re} \,P_k f_1 - \textup{Im} \,h_k(\omega) \textup{Im} \,P_k f_1 \bigr),
\end{align}
and similarly for $f_2^\omega$. In the following we will denote the free wave evolution of the initial data $f^\omega$ by
\begin{align} \label{equ:free_evolution}
 u_f^\omega = \cos(t|\nabla|) f_1^\omega + \frac{\sin(t|\nabla|)}{|\nabla|} f_2^\omega.
\end{align}

\subsection{Statement of the main result}

We are now prepared to state our main result.
\begin{theorem} \label{main_theorem}
 Let $3 < p < 5$ and $\frac{p-1}{p+1} < s < 1$. For real-valued $f = (f_1, f_2) \in H^s_x(\bR^3) \times H^{s-1}_x(\bR^3)$, let $f^\omega = (f_1^\omega, f_2^\omega)$ be the randomized initial data defined in \eqref{equ:bighsrandomization} and let $u_f^\omega$ be the free wave evolution \eqref{equ:free_evolution} of $f^\omega$. Then for almost every $\omega \in \Omega$, there exists a unique global solution
 \begin{equation} \label{equ:solution_main_theorem}
  (u, \partial_t u) \in (u_f^\omega, \partial_t u_f^\omega) + C\bigl(\bR; H^1_x(\bR^3) \times L^2_x(\bR^3)\bigr)
 \end{equation}
 to the nonlinear wave equation
 \begin{equation} \label{equ:nlw_main_theorem}
  \left\{ \begin{aligned}
   -\partial_t^2 u + \Delta u &= |u|^{p-1} u \text{ on } \bR \times \bR^3, \\
   (u, \partial_t u)|_{t=0} &= (f_1^\omega, f_2^\omega).
  \end{aligned} \right.
 \end{equation}
 Here, uniqueness holds in the sense that upon writing 
 \[
  (u, \partial_t u) = (u_f^\omega, \partial_t u_f^\omega) + (v, \partial_t v),
 \]
 there exists a unique global solution 
 \[
  (v, \partial_t v) \in C\bigl(\bR; H^1_x(\bR^3)\bigr) \cap L^{\frac{2p}{p-3}}_{t, loc} L^{2p}_x\bigl(\bR\times\bR^3\bigr) \times C\bigl(\bR; L^2_x(\bR^3)\bigr)
 \]
 to the forced nonlinear wave equation
 \begin{equation} \label{equ:forced_nlw_main_theorem}
  \left\{ \begin{aligned}
   -\partial_t^2 v + \Delta v &= |u_f^\omega + v|^{p-1} (u_f^\omega + v) \text{ on } \bR \times \bR^3, \\
   (v, \partial_t v)|_{t=0} &= (0,0).
  \end{aligned} \right.
 \end{equation}
\end{theorem}

\begin{remark} \label{rem:uniqueness}
In contrast to the mild uniqueness of the authors' previous work \cite[Theorem 1.1]{LM}, Theorem~\ref{main_theorem} yields the more standard notion of uniqueness for solutions to semilinear wave equations, compare with \cite[Remark~4.3]{LM}. Moreover, the threshold for the allowable regularity in Theorem~\ref{main_theorem} has been significantly lowered as compared to that in \cite[Theorem 1.1]{LM}. In particular, we prove the existence of global solutions for initial data at super-critical regularities for all $3 < p < 5$, see Figure~1.1.
\end{remark}

\begin{figure}[hbt] \label{fig:plot}
 \centering
 \includegraphics[height=6.2cm,keepaspectratio]{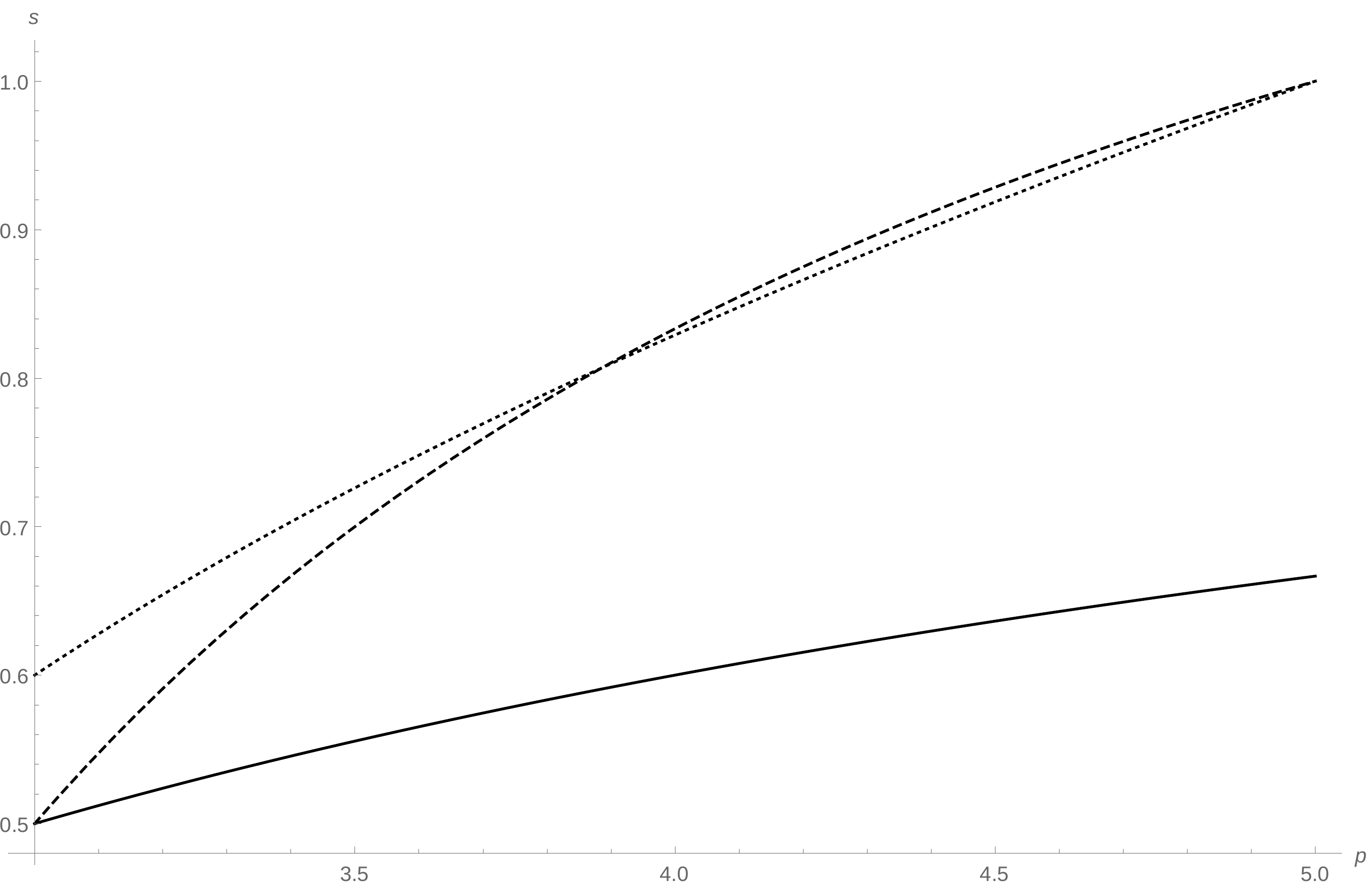}
 \caption{The dashed line is the critical regularity $s_c = \frac{3}{2} - \frac{2}{p-1}$. The solid line is the threshold for the exponent $s$ in Theorem~\ref{main_theorem}. The dotted line is the threshold from the authors' previous result in \cite[Theorem 1.1]{LM}.}
\end{figure} 

While the randomization \eqref{equ:bighsrandomization} does not regularize at the level of Sobolev spaces, the free evolution of the randomized initial data \eqref{equ:free_evolution} almost surely satisfies better space-time integrability properties. For this reason one can show that the nonlinear component of the solution lies in a better space, in this case $H^1_x(\bR^3) \times L^2_x(\bR^3)$, by constructing local solutions via a fixed point argument centered at the free evolution $u_f^\omega$. We will see that to conclude global existence, it suffices to control the growth of the $H^1_x(\bR^3) \times L^2_x(\bR^3)$ norm of the nonlinear component of the solution. The main novelty of this paper is the derivation of probabilistic growth estimates for the modified energy functional
\begin{align}
\label{equ:energy}
  E(v) = \int_{\bR^3} \frac{1}{2} |\nabla_x v|^2 + \frac{1}{2} |\partial_t v|^2 + \frac{1}{2} |v|^2 + \frac{1}{p+1} |u_f^\omega + v|^{p+1} \, dx
\end{align}
for $3 < p < 5$, where $v$ is the nonlinear component of the solution to \eqref{equ:ivp}. Consequently, we will be able to conclude that almost surely, we have the necessary control to extend the local solutions that we construct to global ones.

\medskip

We consider this modified energy functional for two reasons. The first is that the appearance of the free evolution of the randomized initial data in the potential term creates an important cancellation when computing the time derivative of the energy functional. Second, we need the appearance of the $L_x^2$ term in the energy in order to be able to estimate for $0 < \sigma < 1$,
\[
 \big\| |\nabla|^\sigma v \big\|_{L^2_x(\bR^3)}^2 \lesssim \|v\|^2_{L^2_x(\bR^3)} + \|\nabla_x v\|_{L^2_x(\bR^3)}^2 \lesssim E(v).
\]
Previously, energy methods for random data problems were used by Nahmod, Pavlovi\'c and Staffilani \cite{NPS} in the context of the periodic Navier-Stokes equation in two and three dimensions and by Burq and Tzvetkov for the three-dimensional periodic defocusing cubic nonlinear wave equation \cite{BT4}. Pocovnicu~\cite{Pocovnicu} and Oh and Pocovnicu~\cite{OP} used probabilistic energy bounds in conjunction with a probabilistic perturbation theory for the energy critical nonlinear wave equation. 

\begin{remark}
Our proof of the probabilistic energy estimates is inspired by the quintic case in \cite{OP}, with some important differences. In \cite{OP}, Oh and Pocovnicu only consider frequency truncated random initial data and show that almost surely the corresponding solutions satisfy energy bounds \emph{uniformly} in the truncation parameter, which allows one to construct solutions using probabilistic perturbation theory. Instead, we study the Cauchy problem with super-critical random initial data directly. To do so we make use of the observation that although the term $\partial_t u_f^\omega \in H^{s-1}_x(\bR^3)$ appears when taking the time derivative of our energy functional, this expression is always paired with a term at regularity $H^{1-s}_x(\bR^3)$. We therefore have no problem achieving the necessary bounds to close our Gronwall argument, see the proof of Proposition~\ref{prop:energy_bounds} below for more details. Additionally, the presence of non-algebraic nonlinearities introduces some complications in our estimates. To overcome this difficulty, a more careful analysis using the fractional chain rule and interpolation in Sobolev spaces is necessary. 
\end{remark}

\begin{remark}
Our proof does not yield any improvement at $p=3$. However, this case can be treated exactly as in the periodic case in \cite{BT4} using the energy functional
\[
 E(v) = \int_{\bR^3} \frac{1}{2} |\nabla_x v|^2 + \frac{1}{2} |\partial_t v|^2 + \frac{1}{4} |v|^4 \, dx.
\]
One obtains almost sure global existence for any $0 < s < 1$ and uniqueness holds in the same strong sense as in Theorem~\ref{main_theorem}, see Remark~1.5 in \cite{LM}.
\end{remark}

\begin{remark}
To prove scattering of the solutions \eqref{equ:solution_main_theorem} to the random data problem \eqref{equ:nlw_main_theorem}, one needs global control of the $L_t^p L_x^{2p}(\bR \times \bR^3)$ norm of the nonlinear component of the solutions. This will likely require new ideas, which we do not pursue here. A probabilistic version of scattering as in \cite[Theorem 1.4]{BOP2} might be possible, however this only yields scattering on sets of large probability.
\end{remark}

\subsection{Notation}
We denote by $C > 0$ an absolute constant that depends only on fixed parameters and whose value may change from line to line. We write $X \lesssim Y$ if $X \leq C Y$ for some $C > 0$, and analogously for $X \gtrsim Y$. In the sequel, $H^s_x(\bR^3)$, respectively $\dot{H}^s_x(\bR^3)$, denote the usual inhomogeneous, respectively homogeneous, Sobolev spaces. For $s \in \bR$ we define the space 
\[
\cH^s(\bR^3):= H^s_x(\bR^3) \times H^{s-1}_x(\bR^3),
\]
endowed with the obvious norm. We also introduce the shorthand notation $L^\infty_T L^r_x \equiv L^\infty_t L^r_x([0,T]\times\bR^3)$. Finally, for $a \in \bR$ we write $a+$ to denote $a + \varepsilon$ for some arbitrarily small, fixed parameter $\varepsilon > 0$.

\medskip

\noindent {\it  Organization of the paper:} In Section~\ref{sec:preliminaries} we collect several deterministic and probabilistic results. In Section~\ref{sec:proof_main_theorem} we first record a deterministic local well-posedness result for the forced nonlinear wave equation that is associated with the random data problem \eqref{equ:nlw_main_theorem}. Next, we derive key probabilistic energy bounds for the nonlinear components of the solutions to \eqref{equ:nlw_main_theorem}. The proof of Theorem~\ref{main_theorem} is then an immediate consequence.

\medskip

\noindent {\it Acknowledgments:} The authors would like to sincerely thank Michael Eichmair and Gigliola Staffilani for all of their encouragement and support.

\section{Preliminaries} \label{sec:preliminaries}
\setcounter{equation}{0}

\subsection{Deterministic preliminaries}

We begin by recalling the Strichartz estimates for the wave equation, for which we need the following definition.

\begin{definition}
An exponent pair $(q,r)$ is wave-admissible if $2 \leq q \leq \infty$, $2 \leq r < \infty$ and
 \begin{align}
  \frac{1}{q} + \frac{1}{r} \leq \frac{1}{2}.
 \end{align}
\end{definition}

\begin{proposition}[\protect{Strichartz estimates in three space dimensions; \cite{Strichartz}, \cite{Pecher}, \cite{GV2}, \cite{KeelTao}}] \label{prop:strichartz_estimates}
 Suppose $(q, r)$ and $(\tilde{q}, \tilde{r})$ are wave-admissible pairs. Let $u$ be a (weak) solution to the wave equation
 \begin{equation*}
  \left\{ \begin{aligned}
   -\partial_t^2 u + \Delta u \, &= \, h \text{ on } [0, T] \times \bR^3, \\
   (u, \partial_t u)|_{t=0} \, &= \, (f,g)
  \end{aligned} \right.
 \end{equation*}
 for some data $f, g, h$ and time $0 < T < \infty$. Then
 \begin{equation} \label{equ:strichartz_estimates}
  \begin{aligned}
   &\|u\|_{L^q_t L^r_x ([0, T] \times \bR^3)} + \|u\|_{L^{\infty}_t \dot{H}^{\gamma}_x([0,T] \times \bR^3)} + \|\partial_t u\|_{L^{\infty}_t \dot{H}^{\gamma-1}_x([0,T] \times \bR^3)} \\
   &\qquad \qquad \lesssim \, \|f\|_{\dot{H}_x^{\gamma}(\bR^3)} + \|g\|_{\dot{H}_x^{\gamma-1}(\bR^3)} + \|h\|_{L_t^{\tilde{q}'} L^{\tilde{r}'}_x([0,T] \times \bR^3)}
  \end{aligned}
 \end{equation}
 under the assumption that the following scaling conditions hold
 \begin{equation} \label{equ:strichartz_scaling_identities}
  \frac{1}{q} + \frac{3}{r} \, = \, \frac{3}{2} - \gamma \qquad \textup{and} \qquad \frac{1}{\tilde{q}'} + \frac{3}{\tilde{r}'} - 2 \, = \, \frac{3}{2} - \gamma.
 \end{equation} 
\end{proposition}

We say that a wave-admissible pair $(q,r)$ is Strichartz-admissible at regularity $\gamma$ if it satisfies the first identity in \eqref{equ:strichartz_scaling_identities} for some $0 < \gamma < \frac{3}{2}$. 

\medskip

We will need the following two results in order to handle the fractional derivatives which appear in the energy bounds argument in Proposition~\ref{prop:energy_bounds}.

\begin{proposition}[\protect{Fractional chain rule; \cite{CW}}] 
Let $G \in C^1(\bC)$, $\sigma \in (0, 1]$ and suppose that $1 < r, r_1, r_2 < \infty$ satisfy $\frac{1}{r} = \frac{1}{r_1} + \frac{1}{r_2}$. Then
\begin{equation} \label{equ:fractional_chain_rule}
 \| |\nabla|^\sigma G(u) \|_{L^r_x(\bR^3)} \lesssim \| G'(u) \|_{L^{r_1}_x(\bR^3)} \| |\nabla|^\sigma u\|_{L^{r_2}_x(\bR^3)}.
\end{equation}
\end{proposition}

\medskip

\begin{proposition}[\protect{Interpolation estimate; \cite[Section 6.4]{BL}}] 
Let $0 < \theta < 1$, $0 \leq \sigma_0 < \sigma_1$ and $1 < r_0 < r_1 < \infty$. Define $\sigma_\theta$ and $r_\theta$ by
\[
 \sigma_\theta = (1-\theta) \sigma_0 + \theta \sigma_1 \qquad \textup{and} \qquad \frac{1}{r_\theta} = \frac{1-\theta}{r_0} + \frac{\theta}{r_1}. 
\]
Suppose that $u \in \dot{W}^{\sigma_0, r_0}_x(\bR^3) \cap \dot{W}^{\sigma_1, r_1}_x(\bR^3)$. Then $u \in \dot{W}^{\sigma_\theta, r_\theta}_x(\bR^3)$ and
\begin{equation} \label{equ:interpolation_inequality_in_proposition}
 \|u\|_{\dot{W}^{\sigma_\theta, r_\theta}_x(\bR^3)} \lesssim \|u\|^{1-\theta}_{\dot{W}^{\sigma_0, r_0}_x(\bR^3)} \|u\|^\theta_{\dot{W}^{\sigma_1, r_1}_x(\bR^3)}.
\end{equation}
\end{proposition}

\medskip

In the proof of the key energy bounds in Proposition~\ref{prop:energy_bounds}, we shall also need the following Littlewood-Paley projections. Let $\varphi \in C_c^\infty(\bR^3)$ be a radial smooth bump function satisfying $\varphi(\xi) = 1$ for $|\xi| \leq 1$ and $\varphi(\xi) = 0$ for $|\xi| > 2$. We define 
\[
 \widehat{P_1 f}(\xi) := \varphi(\xi) \hat{f}(\xi) 
\]
and for every dyadic $N \geq 2$,
\[
 \widehat{P_N f}(\xi) := \bigl( \varphi(\xi/N) - \varphi(2\xi/N) \bigr) \hat{f}(\xi).
\]
Moreover, we denote by $\widetilde{P}_N$ a fattened Littlewood-Paley projection such that $\widetilde{P}_N P_N = P_N$ for all dyadic $N \geq 1$. 

\subsection{Probabilistic preliminaries}

In this subsection we prove several large deviation estimates for the free evolution of the randomized initial data. In the proof of the key energy bounds in Proposition~\ref{prop:energy_bounds} below, we also need large deviation estimates for the modified free evolution
\begin{equation} \label{equ:modified_free_evolution}
 \tilde{u}_f^\omega := - \frac{|\nabla|}{\langle \nabla \rangle} \sin(t|\nabla|) f_1^\omega + \frac{\cos(t|\nabla|)}{\langle \nabla \rangle} f_2^\omega.
\end{equation}
We note that $\tilde{u}_f^\omega$ satisfies $\partial_t u_f^\omega = \langle \nabla \rangle \tilde{u}_f^\omega$. 

\medskip

The following large deviation estimate is stated for real-valued random variables for simplicity. However, in light of the expression \eqref{equ:real_part} and the independence assumptions on the real and imaginary parts of the random variables used in the randomization, this estimate readily yields the desired results in our setting.

\begin{lemma}[\protect{\cite[Lemma 3.1]{BT1}}] \label{lem:large_deviation_estimate}
 Let $\{h_n\}_{n=1}^{\infty}$ be a sequence of real-valued independent random variables with associated distributions $\{\mu_n\}_{n=1}^{\infty}$ on a probability space $(\Omega, {\mathcal A}, \bP)$. Assume that the distributions satisfy the property that there exists $c > 0$ such that
 \begin{equation*}
  \biggl| \int_{-\infty}^{+\infty} e^{\gamma x} d\mu_n(x) \biggr| \leq e^{c \gamma^2} \text{ for  all } \gamma \in \bR \text{ and for all } n \in \mathbb{N}.
 \end{equation*}
 Then there exists $\alpha > 0$ such that for every $\lambda > 0$ and every sequence $\{c_n\}_{n=1}^{\infty} \in \ell^2(\bN;\bC)$ of complex numbers, 
 \begin{equation*}
  \bP \Bigl( \bigl\{ \omega : \bigl| \sum_{n=1}^{\infty} c_n h_n(\omega) \bigr| > \lambda \bigr\} \Bigr) \leq 2 \exp \biggl(- \alpha \frac{\lambda^2}{\sum_n |c_n|^2} \biggr). 
 \end{equation*}
 As a consequence there exists $C > 0$ such that for every $p \geq 2$ and every $\{c_n\}_{n=1}^{\infty} \in \ell^2(\bN; \bC)$,
 \begin{equation*}
  \Bigl\| \sum_{n=1}^{\infty} c_n h_n(\omega) \Bigr\|_{L^p_\omega(\Omega)} \leq C \sqrt{p} \Bigl( \sum_{n=1}^{\infty} |c_n|^2 \Bigr)^{1/2}.
 \end{equation*}
\end{lemma}

We record the following large deviation estimates for the free evolution of the randomized initial data, following the presentation of the results from \cite{BT4} for the periodic setting. It follows immediately from the estimates below that the free evolution of the randomized initial data satisfies $L_{t,loc}^q L_x^r$ bounds almost surely.

\begin{lemma} \label{lem:large_deviation_estimates_free_evolution}
 Let $\sigma \geq 0$ and $f = (f_1, f_2) \in \cH^\sigma(\bR^3)$. For every $2 \leq q < \infty$, $2 \leq r < \infty$, and $\delta > 1 + \frac{1}{q}$, there exist constants $C \equiv C(q, r, \delta) > 0$ and  $c \equiv c(q, r, \delta) > 0$ such that for every $\lambda > 0$, 
 \begin{equation}
  \bP \Bigl( \bigl\{ \omega \in \Omega : \bigl\|  \langle t \rangle^{-\delta} u_f^\omega \bigr\|_{L^q_t L^r_x(\bR\times\bR^3)} > \lambda \bigr\} \Bigr) \leq C \exp \biggl( - c \frac{\lambda^2}{\|f\|_{\cH^0(\bR^3)}^2} \biggr)
 \end{equation}
 and
 \begin{equation}
  \bP \Bigl( \bigl\{ \omega \in \Omega : \bigl\|  \langle t \rangle^{-\delta} \tilde{u}_f^\omega \bigr\|_{L^q_t L^r_x(\bR\times\bR^3)} > \lambda \bigr\} \Bigr) \leq C \exp \biggl( - c \frac{\lambda^2}{\|f\|_{\cH^0(\bR^3)}^2} \biggr).
 \end{equation}
\end{lemma}
\begin{proof}
 We adapt the proofs of Proposition~A.1, Corollary~A.2, and Corollary~A.4 in \cite{BT4} to our setting. In view of Lemma 2.5 in \cite{LM}, it suffices to prove for any $p \geq q, r$ that
 \[
  \bigl\| \langle t \rangle^{-\delta} u_f^\omega \bigr\|_{L^p_\omega(\Omega; L^q_t L^r_x(\bR\times\bR^3))} \lesssim \sqrt{p} \|f\|_{\cH^0(\bR^3)}
 \]
 and similarly for $\tilde{u}_f^\omega$. To this end we can consider the components of $u_f^\omega$ separately. We will only show
 \begin{equation} \label{equ:suffices}
  \biggl\| \langle t \rangle^{-\delta} \frac{\sin(t|\nabla|)}{|\nabla|} f_2^\omega \biggr\|_{L^p_\omega(\Omega; L^q_t L^r_x(\bR\times\bR^3))} \lesssim \sqrt{p} \|f_2\|_{H^{-1}_x(\bR^3)},
 \end{equation}
since the estimates for the other components are slightly easier. Using Lemma~\ref{lem:large_deviation_estimate} and the unit-scale Bernstein estimate \eqref{equ:unit_scale_bernstein}, we have for any $p \geq q,r$ that
 \begin{align*}
  &\biggl\| \langle t \rangle^{-\delta} \frac{\sin(t|\nabla|)}{|\nabla|} f_2^\omega \biggr\|_{L^p_\omega(\Omega; L^q_t L^r_x(\bR\times\bR^3))} \\
  &= \biggl\| \langle t \rangle^{-\delta} \sum_{k \in \bZ^3} l_k(\omega) \frac{\sin(t|\nabla|)}{|\nabla|} P_k f_2 \biggr\|_{L^p_\omega(\Omega; L^q_t L^r_x(\bR\times\bR^3))} \\
  &\lesssim \sqrt{p} \biggl\| \langle t \rangle^{-\delta} \Bigl( \sum_{k \in \bZ^3} \Bigl| \frac{\sin(t|\nabla|)}{|\nabla|} P_k f_2(x) \Bigr|^2 \Bigr)^{1/2} \biggr\|_{L^q_t L^r_x(\bR\times\bR^3)} \\
  &\lesssim \sqrt{p} \biggl\| \langle t \rangle^{-\delta} \Bigl( \sum_{k \in \bZ^3} \Bigl\| \frac{\sin(t|\nabla|)}{|\nabla|} P_k f_2 \Bigr\|_{L^r_x(\bR^3)}^2 \Bigr)^{1/2} \biggr\|_{L^q_t(\bR)} \\
  &\lesssim \sqrt{p} \biggl\| \langle t \rangle^{-\delta} \Bigl( \sum_{k \in \bZ^3} \Bigl\| \frac{\sin(t|\nabla|)}{|\nabla|} P_k f_2 \Bigr\|_{L^2_x(\bR^3)}^2 \Bigr)^{1/2} \biggr\|_{L^q_t(\bR)} \\
  &\lesssim \sqrt{p} \biggl\| \langle t \rangle^{-(\delta-1)} \Bigl( \sum_{k \in \bZ^3} \bigl\| \langle \nabla \rangle^{-1} P_k f_2 \bigr\|_{L^2_x(\bR^3)}^2 \Bigr)^{1/2} \biggr\|_{L^q_t(\bR)} \\
  &\lesssim \sqrt{p} \bigl\| \langle t \rangle^{-(\delta -1)} \bigr\|_{L^q_t(\bR)}  \Bigl( \sum_{k \in \bZ^3} \bigl\| \langle \nabla \rangle^{-1} P_k f_2 \bigr\|_{L^2_x(\bR^3)}^2 \Bigr)^{1/2}\\
  &\lesssim \sqrt{p} \|f_2\|_{H^{-1}_x(\bR^3)}. \qedhere
 \end{align*}
\end{proof}

The following corollary is the formulation of the large deviation estimates that we will use in the proof of our main result.

\begin{corollary} \label{cor:large_deviation_estimates_free_evolution_with_infty}
 Let $\sigma > 0$ and $f = (f_1, f_2) \in \cH^\sigma(\bR^3)$. For $2 \leq q < \infty$, $2 \leq r \leq \infty$, $\delta > 1 + \frac{1}{q}$, and $0 < \varepsilon \leq \sigma$, there exist constants $C \equiv C(q, r, \delta, \varepsilon) > 0$ and $c \equiv c(q,r,\delta,\varepsilon) > 0$ such that for every $\lambda > 0$,
 \begin{equation}
  \bP \Bigl( \bigl\{ \omega \in \Omega : \bigl\|  \langle t \rangle^{-\delta} u_f^\omega \bigr\|_{L^q_t L^r_x(\bR\times\bR^3)} > \lambda \bigr\} \Bigr) \leq C \exp \biggl( - c \frac{\lambda^2}{\|f\|_{\cH^\varepsilon(\bR^3)}^2} \biggr)
 \end{equation}
 and
 \begin{equation}
  \bP \Bigl( \bigl\{ \omega \in \Omega : \bigl\|  \langle t \rangle^{-\delta} \tilde{u}_f^\omega \bigr\|_{L^q_t L^r_x(\bR\times\bR^3)} > \lambda \bigr\} \Bigr) \leq C \exp \biggl( - c \frac{\lambda^2}{\|f\|_{\cH^\varepsilon(\bR^3)}^2} \biggr).
 \end{equation} 
\end{corollary}
\begin{proof}
 As in \cite[Corollary A.5]{BT4}, the assertion follows immediately from Lemma~\ref{lem:large_deviation_estimates_free_evolution} and the Sobolev embedding $W^{s, r}_x(\bR^3) \hookrightarrow L^\infty_x(\bR^3)$ for any $r > 1$ and $s > \frac{3}{r}$. 
\end{proof}

Finally, we need the following large deviation estimate in order to conclude that the energy functional used in the proof of Proposition~\ref{prop:energy_bounds} is well-defined for all times.
\begin{lemma} \label{lem:large_deviation_estimate_p+1}
Let $\sigma > 0$ and $f = (f_1, f_2) \in \cH^\sigma(\bR^3)$. For $2 \leq r < \infty$, $\delta > 1$, and $0 < \varepsilon \leq \sigma$, there exist constants $C \equiv C(r, \delta, \varepsilon) > 0$ and $c \equiv c(r, \delta, \varepsilon) > 0$ such that for every $\lambda > 0$,
 \[
  \bP \Bigl( \bigl\{ \omega \in \Omega : \bigl\| \langle t \rangle^{-\delta} u_f^\omega \bigr\|_{L_t^\infty L^r_x(\bR \times \bR^3)} > \lambda \bigr\} \Bigr) \leq C \exp \biggl( - c \frac{\lambda^2}{\|f\|_{\cH^\varepsilon(\bR^3)}^2} \biggr).
 \]
\end{lemma}

\begin{proof}
We adapt the proof of Lemma~2.2 in \cite{BTT2} to our setting. Applying one-dimensional Sobolev embedding in time with $q \geq 2$ sufficiently large such that $\varepsilon > \frac{1}{q}$ and $\delta > 1+ \frac{1}{q}$, we obtain
\begin{align*}
 \bigl\| \langle t \rangle^{-\delta} u_f^\omega \bigr\|_{L^\infty_t L^r_x(\bR\times\bR^3)} &\lesssim \bigl\| \langle \partial_t \rangle^{\varepsilon} \langle t \rangle^{-\delta} u_f^\omega \bigr\|_{L^q_t L^r_x(\bR\times\bR^3)} \\
 &\lesssim \bigl\| \langle t \rangle^{-\delta} \langle \partial_t \rangle^{\varepsilon} u_f^\omega \bigr\|_{L^q_t L^r_x(\bR\times\bR^3)} \\
 &\lesssim \bigl\| \langle t \rangle^{-\delta} \langle \nabla_x \rangle^{\varepsilon} u_f^\omega \bigr\|_{L^q_t L^r_x(\bR\times\bR^3)} 
\end{align*}
and the claim now follows from Lemma~\ref{lem:large_deviation_estimates_free_evolution}.
\end{proof}

\section{Proof of Theorem~\ref{main_theorem}} \label{sec:proof_main_theorem}
\setcounter{equation}{0}

In this section we first record a deterministic local well-posedness result for the forced nonlinear wave equation that is associated with the random data problem~\eqref{equ:nlw_main_theorem}. Next we derive the main probabilistic energy bounds for the nonlinear components of the solutions to~\eqref{equ:nlw_main_theorem}. The proof of Theorem~\ref{main_theorem} is then an immediate consequence of these two results.

\begin{lemma} \label{lem:lwp}
 Let $3 < p < 5$, $T' > 0$, and let $(v_1, v_2) \in \cH^{1}(\bR^3)$ and $F \in L^{\frac{2p}{p-3}}_t L^{2p}_x([0,T']\times\bR^3)$ be such that
 \begin{equation}
  \|(v_1, v_2)\|_{ \cH^{1}} +  \bigl\| F \bigr\|_{L^{\frac{2p}{p-3}}_t L^{2p}_x([0,T']\times\bR^3)} \leq \lambda
 \end{equation}
 for some $\lambda > 0$. Then there exists $0 < T \leq T'$ with $T \sim \lambda^{-\frac{2(p-1)}{5-p}}$ and a unique solution
 \[
  (v, \partial_t v) \in  C\bigl([0,T]; H^1_x(\bR^3)\bigr) \cap L^{\frac{2p}{p-3}}_t L^{2 p}_x\bigl([0,T]\times\bR^3\bigr) \times C\bigl([0,T]; L^2_x(\bR^3)\bigr)
 \]
 to the forced nonlinear wave equation
 \begin{equation} \label{equ:nlw_lwp}
  \left\{ \begin{aligned}
   -\partial_t^2 v + \Delta v &= |F+v|^{p-1} (F + v) \text{ on } [0,T] \times \bR^3, \\
   (v, \partial_t v)|_{t=0} &= (v_1,v_2).
  \end{aligned} \right.
 \end{equation}
\end{lemma}

\begin{proof}
We introduce the notation 
\[
 q(p) = \frac{2 p}{p - 3}, \quad \alpha(p) = \frac{5 - p}{2}.
\]
Let $0 < T \leq T'$ to be fixed later. Then we define for $v \in L^{q(p)}_t L^{2p}_x([0,T]\times\bR^3)$ and $t \in [0,T]$,
\begin{equation}
 \Phi(v)(t) = \cos(t|\nabla|) v_1 + \frac{\sin(t|\nabla|)}{|\nabla|} v_2 - \int_0^t \frac{\sin((t-s) |\nabla|)}{|\nabla|} |F + v|^{p -1} (F + v)(s) \, ds.
\end{equation}
We note that the exponent pair $(q(p), 2 p)$ is Strichartz-admissible at regularity $\gamma =1$. Thus, by the Strichartz estimates \eqref{equ:strichartz_estimates} we obtain that
\begin{equation} \label{equ:Phi_estimate1}
 \begin{aligned}  
  \|\Phi(v)\|_{L^{q(p)}_T L^{2 p}_x} &\leq C \Bigl( \|(v_1, v_2)\|_{\dot{H}^1_x \times L^2_x} + \bigl\| |F + v|^{p -1} (F + v) \bigr\|_{L^1_T L^2_x} \Bigr) \\
  &\leq C \|(v_1, v_2)\|_{\dot{H}^1_x \times L^2_x} + C T^{\alpha(p)} \Bigl( \bigl\| F \bigr\|^{p}_{L^{q(p)}_T L^{2 p}_x } + \|v\|^{p}_{L^{q(p)}_T L^{2 p}_x} \Bigr).
 \end{aligned}
\end{equation}
Similarly, we find for $v, \widetilde{v} \in L^{q(p)}_t L^{2p}_x([0,T]\times\bR^3)$ that
\begin{equation} \label{equ:Phi_estimate2}
 \begin{aligned}
   \|\Phi(v) - \Phi(\widetilde{v})\|_{L^{q(p)}_T L^{2 p}_x} &\leq C T^{\alpha(p)} \|v - \widetilde{v}\|_{L^{q(p)}_T L^{2 p}_x } \times \\
   &\quad \quad \quad \times \Bigl(\bigl\| F \bigr\|^{p-1}_{L^{q(p)}_T L^{2 p}_x } + \|v\|^{p-1}_{L^{q(p)}_T L^{2 p}_x } + \|\widetilde{v}\|^{p-1}_{L^{q(p)}_T L^{2 p}_x } \Bigr).
 \end{aligned}
\end{equation}
From \eqref{equ:Phi_estimate1} and \eqref{equ:Phi_estimate2} we conclude that by choosing
\begin{equation}
 T \sim \lambda^{- \frac{(p - 1)}{\alpha(p)}},
\end{equation}
the ball
\[
  B := \bigl\{ v \in L^{q(p)}_t L^{2p}_x([0,T]\times\bR^3) : \|v\|_{L^{q(p)}_t L^{2p}_x([0,T]\times\bR^3)}  \leq 2 C \lambda \bigr\}
\]
is mapped into itself by $\Phi$ and $\Phi$ is a contraction on $B$ with respect to the $L^{q(p)}_T L^{2 p}_x$ norm. Hence, $\Phi$ has a unique fixed point $v$ in $B$ and we easily verify that it satisfies
\[
 (v, \partial_t v) \in C([0,T]; H^1_x(\bR^3) \times L^2_x(\bR^3)),
\]
which finishes the proof.
\end{proof}

\medskip

We now present the proof of the key probabilistic energy bounds.

\begin{proposition} \label{prop:energy_bounds}
 Let $3 < p < 5$, $\frac{p-1}{p+1} < s < 1$ and $0 < \delta < \frac{p+1}{p-1} s -1$. For real-valued $f = (f_1, f_2) \in \cH^s(\bR^3)$, let $f^\omega$ be the associated randomization as defined in \eqref{equ:bighsrandomization} and let $u_f^\omega$, $\tilde{u}_f^\omega$ be the free evolutions \eqref{equ:free_evolution}, respectively \eqref{equ:modified_free_evolution}. Then there exists $\Sigma \subset \Omega$ with $\bP(\Sigma) = 1$ such that for all $\omega \in \Sigma$,
 \begin{align} \label{equ:L_t_loc_almost_surely_energy_bounds}
  \begin{aligned}
   & u_f^\omega \in L_{t,loc}^\infty L^{p+1}_x(\bR \times \bR^3), \\
   & \tilde{u}_f^\omega \in L^{p+1}_{t,loc} L^{p+1}_x(\bR\times\bR^3),
  \end{aligned}
  &&
  \begin{aligned}
   & \langle \nabla \rangle^{\frac{p-1}{2} (1-s+\delta+)} u_f^\omega \in L^4_{t,loc} L^2_x(\bR\times\bR^3), \\
   & \langle \nabla \rangle^{s-\delta} \tilde{u}_f^\omega \in L^2_{t,loc} L^\infty_x(\bR\times\bR^3).
  \end{aligned}
 \end{align}
 For fixed $\omega \in \Sigma$, let $T > 0$ be arbitrary and let  
 \begin{equation} \label{equ:energy_bounds_solution}
  (v^\omega, \partial_t v^\omega) \in C\bigl([0,T]; H^1_x(\bR^3)\bigr) \cap L^{\frac{2p}{p-3}}_t L^{2p}_x\bigl([0,T]\times\bR^3\bigr) \times C\bigl([0,T]; L^2_x(\bR^3)\bigr)
 \end{equation}
 be a solution to
 \begin{equation} \label{equ:nlw_probabilistic_energy_bounds}
  \left\{ \begin{aligned}
   -\partial_t^2 v^\omega + \Delta v^\omega &= |u_f^\omega + v^\omega|^{p-1} (u_f^\omega + v^\omega) \text{ on } [0,T] \times \bR^3, \\
   (v^\omega, \partial_t v^\omega)|_{t=0} &= (v_1, v_2)
  \end{aligned} \right.
 \end{equation} 
 for some $(v_1, v_2) \in \cH^1(\bR^3)$. Then there exists an absolute constant $C > 0$ such that
 \begin{equation}
  \sup_{t \in [0,T]} \bigl\| (v^\omega(t), \partial_t v^\omega(t)) \bigr\|_{\cH^1}^2 \leq C \bigl( \|(v_1, v_2)\|_{\cH^1}^2 + \|v_1\|_{L^{p+1}_x}^{p+1} + \|f_1^\omega\|_{L^{p+1}_x}^{p+1} + B^\omega(T) \bigr) e^{C (T + A^\omega(T))},
 \end{equation}
 where
 \begin{align*}
  A^\omega(T) &= \bigl\| \langle \nabla \rangle^{s-\delta} \tilde{u}_f^\omega \bigr\|_{L^1_t L^\infty_x([0,T]\times\bR^3)}, \\
  B^\omega(T) &= \bigl\| \tilde{u}_f^\omega \bigr\|_{L^{p+1}_t L^{p+1}_x([0,T]\times\bR^3)}^{p+1} + \bigl\| \langle \nabla \rangle^{s-\delta} \tilde{u}_f^\omega \bigr\|_{L^2_t L^\infty_x([0,T]\times\bR^3)}^2 \\
  &\quad + \bigl\| \langle \nabla \rangle^{\frac{p-1}{2} (1-s+\delta+)} u_f^\omega \bigr\|_{L^4_t L^2_x([0,T]\times\bR^3)}^4.
 \end{align*}
\end{proposition}
\begin{proof}
 The large deviation estimates from Lemma~\ref{lem:large_deviation_estimates_free_evolution}, Corollary~\ref{cor:large_deviation_estimates_free_evolution_with_infty}, and Lemma~\ref{lem:large_deviation_estimate_p+1} imply that there exists $\Sigma \subset \Omega$ with $\bP(\Sigma) = 1$ such that \eqref{equ:L_t_loc_almost_surely_energy_bounds} holds for all $\omega \in \Sigma$.
 
 \medskip

 Now for fixed $\omega \in \Sigma$, let $T > 0$ be arbitrary and let
 \[
  (v^\omega, \partial_t v^\omega) \in C\bigl([0,T]; H^1_x(\bR^3)\bigr) \cap L^{\frac{2p}{p-3}}_t L^{2p}_x\bigl([0,T]\times\bR^3\bigr) \times C\bigl([0,T]; L^2_x(\bR^3)\bigr)
 \]
 be a solution to \eqref{equ:nlw_probabilistic_energy_bounds}. For ease of notation we drop the superscript $\omega$ from $v^\omega$ for the remainder of the proof and consider the energy functional
 \[
  E(v(t)) = \int_{\bR^3} \frac{1}{2} |\nabla_x v|^2 + \frac{1}{2} |\partial_t v|^2 + \frac{1}{2} |v|^2 + \frac{1}{p+1} |u_f^\omega + v|^{p+1} \, dx.
 \]
For $\omega \in \Sigma$, this functional is well-defined and finite for all $0 \leq t \leq T$ by Sobolev embedding and the properties from \eqref{equ:L_t_loc_almost_surely_energy_bounds}.  Using that $v$ is a solution to the nonlinear wave equation \eqref{equ:nlw_probabilistic_energy_bounds} and that $\partial_t u_f^\omega = \langle \nabla \rangle \tilde{u}_f^\omega$, we compute 
 \begin{equation} \label{equ:time_derivative_energy}
  \partial_t E(v(t)) = \int_{\bR^3} v \partial_t v \, dx + \int_{\bR^3}  \langle \nabla \rangle \tilde{u}_f^\omega \, |u_f^\omega + v|^{p-1} (u_f^\omega + v) \, dx,
 \end{equation}
 where the second term on the right hand side is to be understood as an $H^{s-1}_x \times H^{1-s}_x$ pairing. The first term can be easily bounded by
 \begin{equation}
  \Big| \int_{\bR^3} v \partial_t v \, dx \Big| \lesssim \|v\|_{L^2_x} \|\partial_t v\|_{L^2_x} \lesssim E(v).
 \end{equation}
 In order to estimate the second term on the right hand side of \eqref{equ:time_derivative_energy}, we use an inhomogeneous dyadic decomposition to expand
 \begin{align*}
  \int_{\bR^3}  \langle \nabla \rangle \tilde{u}_f^\omega \, |u_f^\omega + v|^{p-1} (u_f^\omega + v) \, dx &= \sum_{N \geq 1} \int_{\bR^3} \widetilde{P}_N \bigl( \langle \nabla \rangle \tilde{u}_f^\omega \bigr) P_N \bigl( |u_f^\omega + v|^{p-1} (u_f^\omega + v) \bigr) \, dx.
 \end{align*}
 {\bf Case 1:} $N \lesssim 1$. By Bernstein's and Young's inequality
 \begin{align*}
  \bigg| \sum_{N \lesssim 1} \int_{\bR^3} \widetilde{P}_N \bigl( \langle \nabla \rangle \tilde{u}_f^\omega \bigr) P_N \bigl( |u_f^\omega + v|^{p-1} (u_f^\omega + v) \bigr) \, dx \bigg| &\lesssim \bigl\| \tilde{u}_f^\omega \bigr\|_{L^{p+1}_x} \bigl\| |u_f^\omega + v|^{p-1} (u_f^\omega +v) \bigr\|_{L^{\frac{p+1}{p}}_x} \\
  &\lesssim \bigl\| \tilde{u}_f^\omega \bigr\|_{L^{p+1}_x}^{p+1} + \bigl\| u_f^\omega + v \bigr\|_{L^{p+1}_x}^{p+1} \\
  &\lesssim \bigl\| \tilde{u}_f^\omega \bigr\|_{L^{p+1}_x}^{p+1} + E(v).
 \end{align*}
 {\bf Case 2:} $N \gg 1$. We further decompose dyadically
 \begin{align*}
  &\sum_{N \gg 1} \int_{\bR^3} \widetilde{P}_N \bigl( \langle \nabla \rangle \tilde{u}_f^\omega \bigr) P_N \bigl( |u_f^\omega + v|^{p-1} (u_f^\omega + v) \bigr) \, dx \\
  &= \sum_{N \gg 1} \sum_{N_1, N_2 \geq 1} \int_{\bR^3} \widetilde{P}_N \bigl( \langle \nabla \rangle \tilde{u}_f^\omega \bigr) P_N \Bigl( P_{N_1} \bigl( |u_f^\omega + v|^{p-1} \bigr) P_{N_2} \bigl( u_f^\omega + v \bigr) \Bigr) \, dx.
 \end{align*}
 {\bf Case 2a:} $N_1 \sim N$ and $N_2 \lesssim N_1$. Summing over dyadic $N, N_2$ at a slight loss of a power in $N_1$, we obtain
 \begin{align*}
 &\bigg| \sum_{N \gg 1} \sum_{N_1 \sim N} \sum_{N_2 \lesssim N_1} \int_{\bR^3} {\widetilde P}_N \big( \langle \nabla \rangle \tilde{u}_f^\omega \big) P_N \Big( P_{N_1} \big( |u_f^\omega + v|^{p-1} \big) P_{N_2} \big( u_f^\omega + v \big) \Big) \, dx \bigg| \\
 &\lesssim \big\| \langle \nabla \rangle^{s-\delta} \tilde{u}_f^\omega \big\|_{L^\infty_x} \sup_{N_1 \gg 1} N_1^{1-s+\delta +} \big\| P_{N_1} |u_f^\omega + v|^{p-1} \big\|_{L^{\frac{p+1}{p}}_x} \big\| u_f^\omega + v \big\|_{L^{p+1}_x} \\
 &\lesssim \big\| \langle \nabla \rangle^{s-\delta} \tilde{u}_f^\omega \big\|_{L^\infty_x} \big\| |\nabla|^{1-s+\delta+} \big( |u_f^\omega + v|^{p-1} \big) \big\|_{L^{\frac{p+1}{p}}_x} \big\| u_f^\omega + v \big\|_{L^{p+1}_x}. 
 \end{align*}
 Using the fractional chain rule \eqref{equ:fractional_chain_rule} with $\frac{p}{p+1} = \frac{2}{p+1} + \frac{p-2}{p+1}$, this is bounded by
 \begin{equation} \label{equ:bound_on_difficult_term}
  \big\| \langle \nabla \rangle^{s-\delta} \tilde{u}_f^\omega \big\|_{L^\infty_x} \big\| |\nabla|^{1-s+\delta+} (u_f^\omega + v) \big\|_{L^{\frac{p+1}{2}}_x} \big\| u_f^\omega + v \big\|^{p-1}_{L^{p+1}_x}.
 \end{equation} 
We invoke the interpolation inequality \eqref{equ:interpolation_inequality_in_proposition} in the form 
 \begin{equation} \label{equ:interpolation_inequality}
  \bigl\| |\nabla|^\sigma f \bigr\|_{L^{\frac{p+1}{2}}_x} \lesssim \bigl\| |\nabla|^{\frac{p-1}{2} \sigma} f \bigr\|_{L^2_x}^{\frac{2}{p-1}} \bigl\|f\bigr\|_{L^{p+1}_x}^{\frac{p-3}{p-1}}
 \end{equation}
 for $0 < \sigma < 1$ to bound \eqref{equ:bound_on_difficult_term} by
 \begin{align*}
  \big\| \langle \nabla \rangle^{s-\delta} \tilde{u}_f^\omega \big\|_{L^\infty_x} \big\| |\nabla|^{\frac{p-1}{2} (1-s+\delta+)} (u_f^\omega + v) \big\|_{L^2_x}^{\frac{2}{p-1}} \big\| u_f^\omega + v \big\|_{L^{p+1}_x}^{\frac{(p-2)(p+1)}{p-1}},
 \end{align*}
which we can estimate using Young's inequality
 \begin{equation} \label{equ:difficult_term_after_young}
  \begin{split}
   &\big\| \langle \nabla \rangle^{s-\delta} \tilde{u}_f^\omega \big\|_{L^\infty_x} \Big( \big\| |\nabla|^{\frac{p-1}{2} (1-s+\delta+)} (u_f^\omega + v) \big\|_{L^2_x}^2 + \big\| u_f^\omega + v \big\|_{L^{p+1}_x}^{p+1} \Big) \\
   &\lesssim \big\| \langle \nabla \rangle^{s-\delta} \tilde{u}_f^\omega \big\|_{L^\infty_x} \Big( \big\| |\nabla|^{\frac{p-1}{2} (1-s+\delta+)} u_f^\omega \big\|_{L^2_x}^2 + \big\| |\nabla|^{\frac{p-1}{2} (1-s+\delta+)} v \big\|_{L^2_x}^2 + \big\| u_f^\omega + v \big\|_{L^{p+1}_x}^{p+1} \Big).
  \end{split}
 \end{equation}
By our assumptions on $s$ and $\delta$, we have that
 \[
  \frac{p-1}{2} (1 - s + \delta +) < s < 1,
 \]
 thus, we may bound the right hand side of \eqref{equ:difficult_term_after_young} by
 \begin{align*}
  \big\| \langle \nabla \rangle^{s-\delta} \tilde{u}_f^\omega \big\|_{L^\infty_x} \Big( \big\| \langle \nabla \rangle^{\frac{p-1}{2} (1-s+\delta+)} u_f^\omega \big\|_{L^2_x}^2 + \big\| v \big\|_{L^2_x}^2 + \big\| \nabla_x v \big\|_{L^2_x}^2 + \big\| u_f^\omega + v \big\|_{L^{p+1}_x}^{p+1} \Big).
 \end{align*}
 {\bf Case 2b:} $N_2 \sim N$ and $N_1 \lesssim N_2$. Summing over dyadic $N$, $N_1$, we bound
 \begin{align*}
  &\bigg| \sum_{N \gg 1} \sum_{N_2 \sim N} \sum_{N_1 \lesssim N_2} \int_{\bR^3} {\widetilde P}_N \big( \langle \nabla \rangle \tilde{u}_f^\omega \big) P_N \Big( P_{N_1} \big( |u_f^\omega + v|^{p-1} \big) P_{N_2} \big( u_f^\omega + v \big) \Big) \, dx \bigg| \\
  &\lesssim \big\| \langle \nabla \rangle^{s-\delta} \tilde{u}_f^\omega \big\|_{L^\infty_x} \sup_{N_2 \gg 1} N_2^{1-s+\delta+} \big\| u_f^\omega + v \big\|^{p-1}_{L^{p+1}_x} \big\| P_{N_2} (u_f^\omega + v) \big\|_{L^{\frac{p+1}{2}}_x}.
 \end{align*}  
 Using the interpolation inequality
 \[
  \|f\|_{L^{\frac{p+1}{2}}_x} \lesssim \|f\|_{L^2_x}^{\frac{2}{p-1}} \|f\|_{L^{p+1}_x}^{\frac{p-3}{p-1}},
 \]
 this is bounded by
 \begin{align*}
  &\big\| \langle \nabla \rangle^{s-\delta} \tilde{u}_f^\omega \big\|_{L^\infty_x} \sup_{N_2 \gg 1} N_2^{1-s+\delta+} \big\| u_f^\omega + v \big\|^{p-1}_{L^{p+1}_x} \big\| P_{N_2} (u_f^\omega + v) \big\|_{L^2_x}^{\frac{2}{p-1}} \big\| P_{N_2} (u_f^\omega + v) \big\|_{L^{p+1}_x}^{\frac{p-3}{p-1}} \\
  &\lesssim \big\| \langle \nabla \rangle^{s-\delta} \tilde{u}_f^\omega \big\|_{L^\infty_x} \sup_{N_2 \gg 1} \Big\| N_2^{\frac{p-1}{2} (1-s+\delta+)} P_{N_2} (u_f^\omega + v) \Big\|_{L^2_x}^{\frac{2}{p-1}} \big\| u_f^\omega + v \big\|_{L^{p+1}_x}^{\frac{(p-2)(p+1)}{p-1}}.
 \end{align*}
By Bernstein's and Young's inequality we finally obtain the estimate
 \begin{align*}
  &\big\| \langle \nabla \rangle^{s-\delta} \tilde{u}_f^\omega \big\|_{L^\infty_x} \Big( \big\| |\nabla|^{\frac{p-1}{2} (1-s+\delta+)} (u_f^\omega + v) \big\|_{L^2_x}^2 + \big\| u_f^\omega + v \big\|_{L^{p+1}_x}^{p+1} \Big) \\
  &\lesssim \big\| \langle \nabla \rangle^{s-\delta} \tilde{u}_f^\omega \big\|_{L^\infty_x} \Big( \big\| \langle \nabla \rangle^{\frac{p-1}{2} (1-s+\delta+)} u_f^\omega \big\|_{L^2_x}^2 + \big\| v \big\|_{L^2_x}^2 + \big\| \nabla_x v \big\|_{L^2_x}^2 + \big\| u_f^\omega + v \big\|_{L^{p+1}_x}^{p+1} \Big) \\
  &\lesssim \big\| \langle \nabla \rangle^{s-\delta} \tilde{u}_f^\omega \big\|_{L^\infty_x}^2 + \big\| \langle \nabla \rangle^{\frac{p-1}{2} (1-s+\delta+)} u_f^\omega \big\|_{L^2_x}^4 \\
 &\qquad + \big\| \langle \nabla \rangle^{s-\delta} \tilde{u}_f^\omega \big\|_{L^\infty_x} \Big( \big\| v \big\|_{L^2_x}^2 + \big\| \nabla_x v \big\|_{L^2_x}^2 + \big\| u_f^\omega + v \big\|_{L^{p+1}_x}^{p+1} \Big).
 \end{align*}
 {\bf Case 2c:} $N_1, N_2 \gtrsim N$ and $N_1 \sim N_2$. Here we can proceed similarly to Case 2b.

 \medskip

 Putting the above estimates together, we find that for $0 \leq t \leq T$,
 \begin{align*}
  \partial_t E(v(t)) &\leq C \Big( \big\| \tilde{u}_f^\omega(t) \big\|_{L^{p+1}_x}^{p+1} + \big\| \langle \nabla \rangle^{s-\delta} \tilde{u}_f^\omega(t) \big\|_{L^\infty_x}^2 + \big\| \langle \nabla \rangle^{\frac{p-1}{2} (1-s+\delta+)} u_f^\omega(t) \big\|_{L^2_x}^4 \\
  &\quad \quad \quad \quad + \big( 1 + \big\| \langle \nabla \rangle^{s-\delta} \tilde{u}_f^\omega(t) \big\|_{L^\infty_x} \big) E(v(t)) \Big).
 \end{align*}
 Gronwall's inequality then yields that
 \begin{align*}
  E(v(T)) &\leq \exp \biggl( C \int_0^T \bigl( 1 + \bigl\| \langle \nabla \rangle^{s-\delta} \tilde{u}_f^\omega(t) \bigr\|_{L^\infty_x} \bigr) \, dt \biggr) \times \\
  &\quad \times \biggl( E(0) + C \int_0^T \Bigl( \bigl\| \tilde{u}_f^\omega(t) \bigr\|_{L^{p+1}_x}^{p+1} + \bigl\| \langle \nabla \rangle^{s-\delta} \tilde{u}_f^\omega(t) \bigr\|_{L^\infty_x}^2 + \bigl\| \langle \nabla \rangle^{\frac{p-1}{2} (1-s+\delta+)} u_f^\omega(t) \bigr\|_{L^2_x}^4 \Bigr) \, dt \biggr) \\
  &\leq C e^{C (T + A^\omega(T))} \Bigl( \|(v_1, v_2)\|_{\cH^1}^2 + \|v_1\|_{L^{p+1}_x}^{p+1} + \| f_1^\omega \|_{L^{p+1}_x}^{p+1} + B^\omega(T) \Bigr). \qedhere
 \end{align*}
\end{proof}

Theorem~\ref{main_theorem} is now an easy consequence of the previous results.
\begin{proof}[Proof of Theorem~\ref{main_theorem}]
 The large deviation estimates from Lemma~\ref{lem:large_deviation_estimates_free_evolution}, Corollary~\ref{cor:large_deviation_estimates_free_evolution_with_infty}, and Lemma~\ref{lem:large_deviation_estimate_p+1} imply that there exists $\widetilde{\Omega} \subset \Omega$ with $\bP(\widetilde{\Omega}) = 1$ such that for all $\omega \in \widetilde{\Omega}$, \eqref{equ:L_t_loc_almost_surely_energy_bounds} holds and 
 \[
  u_f^\omega \in L^{\frac{2p}{p-3}}_{t,loc} L^{2p}_x(\bR\times\bR^3).
 \]
 Fix $\omega \in \widetilde{\Omega}$. We invoke Lemma~\ref{lem:lwp} to obtain a local solution $(v^\omega, \partial_t v^\omega)$ to \eqref{equ:forced_nlw_main_theorem}. We can iterate the local well-posedness of Lemma~\ref{lem:lwp} to obtain a global solution provided the $\cH^1(\bR^3)$ norm of $(v^\omega(t), \partial_t v^\omega(t))$ does not blow up in finite time. This is guaranteed by the energy bounds of Proposition~\ref{prop:energy_bounds}, which proves the assertion.
\end{proof}

\bibliographystyle{myamsplain}
\bibliography{references}

\end{document}